\documentclass{article}
\usepackage[T1]{fontenc}

\usepackage{amsthm}
\usepackage{amsmath}
\usepackage{amssymb}
\usepackage{wrapfig}
\usepackage[lofdepth,lotdepth]{subfig}

\input{insbox}

\usepackage{hyperref}
\hypersetup{
    colorlinks=true,
    linkcolor=black,
    filecolor=magenta,
    citecolor = red,      
    urlcolor=cyan,
    pdftitle={Overleaf Example},
    pdfpagemode=FullScreen,
    }

\usepackage{wrapfig}
\usepackage{cutwin}

\usepackage{tkz-euclide}
\usepackage{pgf,tikz}
\usetikzlibrary{decorations.markings}
\usepackage{mathrsfs}
\usetikzlibrary{arrows}

\usepackage{multirow}

\usepackage{enumerate}
\usepackage{url}

\usepackage[framemethod=tikz]{mdframed}
\usepackage[a4paper, top = 20mm, left = 25mm, right = 25mm]{geometry}

\newtheorem{theorem}{Theorem}
\newtheorem*{corollary}{Corollary}
\newtheorem*{conj}{Conjecture}

\newcommand{\R}{\mathbb{R}}

\title{Constructing flexible polyhedra by twinning}
\author{Elvar Atlason and Simon Guest}
\date{October 2025}

\begin{document}

\maketitle

\begin{abstract} 
Polyhedra are generically rigid, but can be made to flex under certain symmetry conditions. We generalise Raoul Bricard's 1897 method for making flexible octahedra to construct an infinite family of flexible polyhedra with self-intersections. Removing an edge from any of these models gives a crinkle, and these can be used to create flexible polyhedra without self-interesection. We show this in a particular example, giving a flexible embedded polyhedron with a large range of motion. We also discuss a novel crinkle.
\end{abstract}

\tableofcontents

\section*{Introduction}

\begin{figure}
\centering
\includegraphics[width=0.7\textwidth]{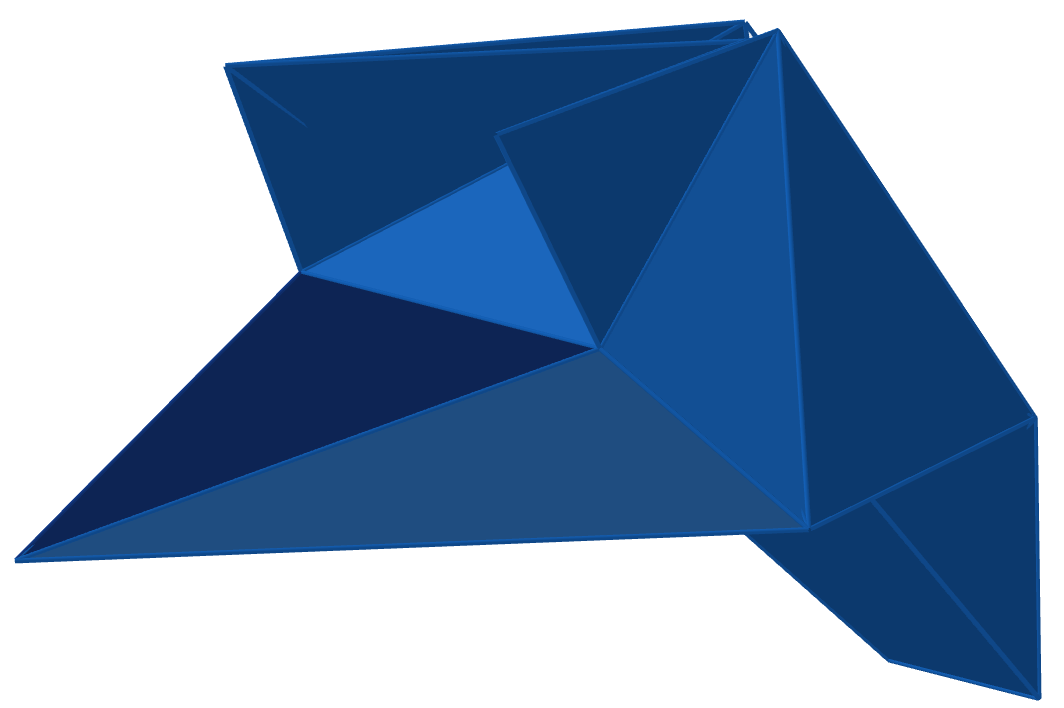}
\caption{An embedded flexible polyhedron, using a pentagonal crinkle arising from a twinned digonal anticupola. A net is given in figure \ref{net}, and a model is available online at \url{https://www.geogebra.org/m/fq3vdnmy}.}
\label{foxtrot}
\end{figure}

If you make a paper model of any convex polyhedron, it will not flex, despite the fact that each edge functions as a hinge. A simple counting argument shows that a generic polyhedron is rigid, and it was a longstanding conjecture that any embedded polyhedron would also be rigid. 

The rigidity conjecture for polyhedra was disproved by Robert Connelly in 1977, after he explicitly described an embedded flexible polyhedron in  \cite{Connelly_original}. The result attracted much interest, and the methods used would prove useful in the design of deployable structures. A discussion of some applications and relevance to engineering is given in the book \cite{SimonBook}. In the thesis \cite{Li}, Lijingjiao lists some known results, showing that not many examples of flexible polyhedra are known. Below, we will generalise a construction of Raoul Bricard from 1897, \cite{Bricard}, and introduce an infinite family of flexible polyhedra with self-intersections. This produces novel crinkles of infinitely many different topologies, and Connelly's methods can be applied to those to construct embedded polyhedra. We will do this in a particular case in section \ref{foxtrot_section}, leading to the fox-like flexible polyhedron shown in figure \ref{foxtrot}.

This paper is split into 3 sections. In section \ref{section_rigidity} we discuss the rigidity phenomenon for polyhedra. In section \ref{section_twinning} we describe the twinning method for constructing infinitely many new flexible polyhedra. In section \ref{section_examples} we work through some simple examples of twinned polyhedra.

\section{Rigidity}
\label{section_rigidity}

We say that a polyhedron can \textit{flex} if there is a non-trivial transformation of the vertices keeping all of the edge lengths fixed. Note that simply translating or rotating the polyhedron in $\R^3$ does not count as a flex.

The number of degrees of freedom of a collection of $V$ vertices connected in a generic way by $E$ bars in $\R^3$ is $3V-E$. Each vertex adds three degrees of freedom and each edge removes one. Note that if the edges are not aligned generically, there may be some redundant edges in the arrangement that fail to remove a degree of freedom. In a triangulated polyhedron, we have $V-E+F = 2$ and $3F= 2E$. This implies that the number of degrees of freedom of the polyhedron is $$3V-E = 6.$$ A rigid body in $\R^3$ has $3+3 = 6$ degrees of freedom, corresponding to translation in three independent directions and rotation in three independent angles. This heuristic count of the degrees of freedom suggests that a triangulated polyhedron is rigid. For a more rigorous version of this counting argument, see \cite{Gluck_genericity}.

The rigidity conjecture for polyhedra is already suggested by definition ten of book XI of Euclid's Elements\footnote{Equal and similar solid figures are those contained by similar planes equal in multitude and magnitude, \cite{EuclidsElementsHeath}.}. No doubt the question of rigidity held practical value in ancient Greece, as it does now to carpenters. Euclid is interested in rigidity, demonstrated by the proof of I.VIII, and VI.def I, but he deals with the topic in more detail in two dimensions. Later scholars have criticised his treatment of the problem of rigidity in three dimension, that of XI.def X, on account of it deserving the role of a proposition to be proved, and also because it turned out to be false. The definition he gives suggests an anticipation of the generic rigidity of polyhedra, and an understanding of how it differs from the two dimensional case. A further discussion of the work of Euclid can be found in Thomas L. Heath's notes in his translation, \cite{EuclidsElementsHeath}.

In a paper from 1813, \cite{Cauchys_rigidity}, Augustin Cauchy showed that convex polyhedra are rigid. See also \cite{Cauchys_Theorem_Proof} for a simple proof of this fact. However, in the paper \cite{Bricard}, Raoul Bricard described families of self-intersecting flexible octahedra that could be constructed out of rods. In 1977, Robert Connelly found a flexible embedded polyhedron. He did this by assembling Bricard octahedra to remove the self-intersections, see \cite{p1}. Further simplification by Klaus Steffen led to a simple working example of 9 vertices, discussed in \cite{Connelly_note}. 

Later analysis of flexible polyhedra has shown that during flexion, their volume remains constant. This was called the bellows conjecture, stating that a flexible polyhedron could not be used as bellows, and was first proved in \cite{Bellows_conjecture}. The proof constructs a polynomial from the combinatorial structure of the polyhedron, such that the volume of the polyhedron is a root. Since the roots of a polynomial form a discrete set, the volume must be constant while the polyhedron goes through a continuous flexing motion. This greatly restricts the type of motion that a flexible polyhedron can undergo.

Bricard's original octahedra come in three types. Type I and II use general symmetry arguments, whereas type III is more involved, relying on a ruling of the hyperboloid. We will focus on the first two types here. A discussion of the type III Bricard octahedron can be found in \cite{BricardIII_Guest}. Generalising the construction of Bricard's type I and II octahedra, we introduce a method called \textit{twinning} to create an infinite family of self-intersecting flexible polyhedra. Connelly's methods can then be applied to make models without self-intersections. Using the twinning construction, we explicitly construct a new variation of crinkle, as well as a novel flexible polyhedron without self-intersections. The models described have a large range of motion and show some interesting properties, which could be of use in structural engineering.

\section{Twinning}
\label{section_twinning}
We generalise Bricard's flexible octahedra of types I and II. For any triangulated polyhedron $P$ containing a symmetric quadrilateral $ABA'B'$, we remove the edge $AA'$ to create a polyhedral cap on the quadrilateral base $ABA'B'$. If the polyhedron $P$ was previously rigid, this cap will have exactly one degree of freedom. We then make two copies of the cap, rotate one of them, and glue the two copies together along their common boundary to create a twinned polyhedron $T$.

\subsection{Type I}
The type I Bricard octahedron relies on the following theorem.

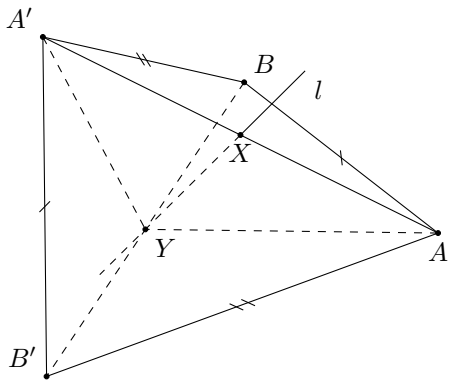
\begin{figure}\centering
\begin{tikzpicture}[line join=bevel,scale = 1.3,tick/.style={
    postaction={decorate},
    decoration={
      markings,
      mark=at position 0.5 with {
        \draw[-] (-2pt,2pt) -- (2pt,-2pt);
      }
    }
  }, doubletick/.style={
    postaction={decorate},
    decoration={
      markings,
      mark=at position 0.485 with {
        \draw[-] (-2pt,2pt) -- (2pt,-2pt);
      },  mark=at position 0.515 with {
        \draw[-] (-2pt,2pt) -- (2pt,-2pt);
      }
    }
  }]
\coordinate (B) at (1,1.5,1.5);
\coordinate (A) at (2,-1,-1);
\coordinate(B') at (-1,-1.5,1.5);
\coordinate (A') at (-2,1,-1);
\coordinate(Y) at (0,0,1.5);
\coordinate(X) at (0,0,-1);
\coordinate(X') at (0,0,2.7);
\coordinate(Y') at (0,0,-2.7);
\draw[tick](A)--(B);
\draw[doubletick](A')--(B);
\draw[tick](A')--(B');
\draw[doubletick](B')--(A);
\draw[dashed](B)--(B');
\draw(A)--(A');
\draw[dashed](A')--(Y);
\draw[dashed](A)--(Y);
\draw[fill=black] (A) circle [radius=0.7pt] node[below] {$A$};
\draw[fill=black] (B) circle [radius=0.7pt] node[above right] {$B$};
\draw[fill=black] (A') circle [radius=0.7pt] node[above left] {$A'$};
\draw[fill=black] (B') circle [radius=0.7pt] node[above left] {$B'$};
\draw[fill=black] (X) circle [radius=0.7pt] node[below] {$X$};
\draw[fill=black] (Y) circle [radius=0.7pt] node[below right] {$Y$};
\draw[dashed] (X)--(X');
\draw(X)--(Y');
\draw (Y') node[ below right]{$l$};
\end{tikzpicture}
\caption{The quadrilateral $ABA'B'$ has rotational symmetry.}
\label{theorem1_fig}
\end{figure}

\begin{theorem}
\label{line_symmetry}
Let $A$, $B$, $A'$, $B'$ be four points in $\mathbb{R}^3$, such that $AB = A'B'$ and $AB' = A'B$. Then there exists a line $l$ in $\mathbb{R}^3$ such that a half-rotation in $l$ swaps $A$ with $A'$ and $B$ with $B'$.
\end{theorem}
\begin{proof} Let the four points $A$, $B$, $A'$, $B'$ be as described above.

In the case when the diagonals $AA'$ and $BB'$ intersect, the quadrilateral $ABA'B'$ is planar, and thus a parallelogram. It has a rotational symmetry in a line perpendicular to the plane of the four points.

Assuming the diagonals do not intersect, let $X$ be the midpoint of $AA'$ and $Y$ be the midpoint of $BB'$, as shown in figure \ref{theorem1_fig}. Let $l$ be the line through $X$ and $Y$.

Since $\triangle{ABB'} \cong \triangle{BA'B'}$, and $Y$ is the midpoint of $BB'$, we have $$AY = A'Y,$$ since they are corresponding medians in congruent triangles.  Therefore, we have $\triangle{AXY} \cong \triangle{A'XY}$, since the corresponding sides in these triangles are equal. In particular, $$\angle{AXY} = \angle{A'XY},$$ and since $A',X,A$ lie on a line, we have $$XY \perp AA'.$$

Similarly, we obtain $XY \perp BB',$ and so it follows that $l$ is the line of symmetry of the quadrilateral $ABA'B'$.
\end{proof}

Let $ABA'B'$ be a quadrilateral in $\R^3$ such that the edges satisfy $AB = A'B'$ and $AB' = A'B$, as in the theorem above. To construct the Bricard octahedron of type I, call the line of symmetry $l$, and pick some point $C$ not on $l$. We then draw the edges $CA,CB,CA',CB'$, to create a pyramid with a missing bottom. This is a chain of four triangles, and it has a single degree of freedom. Throughout the flexing motion, the base $ABA'B'$ of the pyramid has a line of symmetry, $l$. If we rotate the pyramid in that line, we send $C$ to $C'$, $A$ to $A'$ and $B$ to $B'$. Note that by construction, the base is fixed by this rotation. So we can glue the rotated pyramid onto the original pyramid to obtain a closed octahedron. This is a flexible polyhedron with self-intersection, the first of three types of flexible octahedra described by Raoul Bricard in 1897, \cite{Bricard}. See a depiction in figure \ref{Bricard octahedron}. A model is also available at \url{https://www.geogebra.org/m/pshaep9q}.

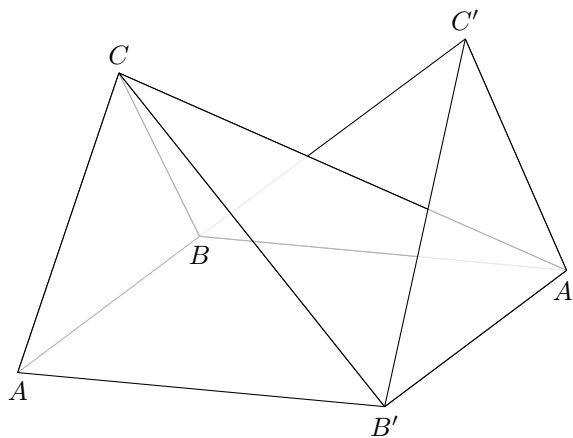
\begin{figure}
\centering
\begin{tikzpicture}[line join = bevel,scale = 0.55\textwidth/3cm]
\coordinate (A1) at (-0.8,0,-1);
\coordinate (A2) at (-1,0,0.6);
\coordinate (A3) at (0.8,0,1);
\coordinate (A4) at (1,0,-0.6);
\coordinate (B1) at (-0.7,1.2,0.2);
\coordinate (C1) at (0.7,1.2,-0.2);

\draw [fill opacity=1,fill=white](A1) -- (A2) -- (B1) -- cycle;
\draw [fill opacity=1,fill=white](A4) -- (A1) -- (C1) -- cycle;
\draw [fill opacity=0.7,fill=white] (A4) -- (A1) -- (B1) -- cycle;
\draw [fill opacity=0.7,fill=white] (A2) -- (A3) -- (B1) -- cycle;
\draw [fill opacity=0.7, fill=white] (A3) -- (A4) -- (B1) -- cycle;
\draw [fill opacity=0.7,fill=white] (A3) -- (A4) -- (C1) -- cycle;

\node[below] at (A1) {$B$};
\node[below] at (A2) {$A$};
\node[below] at (A3) {$B'$};
\node[below] at (A4) {$A'$};
\node[above] at (B1) {$C$};
\node[above] at (C1) {$C'$};
\end{tikzpicture}
\caption{The Bricard octahedron, or twinned pyramid. Two faces have been omitted to remove the self-inersections, creating a crinkle.}
\label{Bricard octahedron}
\end{figure}

More generally, let $P$ be a rigid polyhedron with two faces $\triangle{ABA'}$ and $\triangle{A'B'A}$ along the edge $AA'$, such that $AB = A'B'$ and $AB' = BA'$. The computation in the introduction shows that for rigidity to occur, every edge must restrict the motion. So if we remove the edge $AA'$, the resulting polyhedral cap on a quadrilateral hole has one degree of freedom. Throughout the flexing motion of this cap, the boundary $ABA'B'$ retains a rotational symmetry in some line $l$ by theorem \ref{line_symmetry}. Let $P_{cap}$ be the polyhedral cap obtained from $P$ after removing the edge $AA'$. Let $P'_{cap}$ be a copy of the polyhedral cap $P_{cap}$, obtained by rotating the polyhedral cap $P_{cap}$ by $180^{\circ}$ around the line $l$. By construction, the boundary of $P_{cap}$ aligns with the boundary of $P'_{cap}$, so we can glue $P_{cap}$ to $P'_{cap}$ along the boundary, obtaining a flexible polyhedron, $T$. The construction for a type I twinning can be realised in the following geogebra file, \url{https://www.geogebra.org/m/vabhua2b}, here, by letting the slider \textit{distance} go to zero, the two caps are glued together along their common quadrilateral bases.

A quadrilateral in $\R^3$ has $3\cdot 4-4 = 8$ degrees of freedom, so disregarding the Euclidian symmetry group, there are two flexing degrees of freedom. The hole we create in $P$ by removing a single edge has one degree of freedom, so it may be regarded as a path in the two-dimensional space of all realisations of the quadrilateral. In other words, the polyhedral cap on top of the quadrilateral restricts it from having two degrees of freedom to just one. By copying the cap and rotating it, we obtain a different cap on the quadrilateral base which has exactly the same movement, allowing us to glue the two caps together to obtain a closed polyhedron.

We call $ABA'B'$ the \textit{equator} of $T$, and we call this action \textit{twinning}\footnote{The wording is borrowed from crystallography, a twinned crystal is a pair of crystals growing in different directions from the same base.}. We call $T$ the \textit{twinned} polyhedron of the original polyhedron $P$ around the equator $ABA'B'$. 

Let us summarise this in a theorem.

\begin{theorem}
\label{formal_statement_I}
Let $P$ be a polyhedron such that the edges form an isostatic framework. Assume that there are some faces $\triangle{ABA'}$ and $\triangle{BA'B'}$ in $P$, such that $AB = A'B'$ and $AB' = BA'$. Let $T$ be a twinned polyhedron of $P$ around the equator $ABA'B'$. Then $T$  has a single flex.
\end{theorem}

\begin{corollary}
If $P$ is convex, the twinned polyhedron $T$ has a single flex.
\end{corollary}

The polyhedron $P$ is not a generic polyhedron as we have restricted some edge lengths to be equal, however, we do not expect that to affect the generic case. We therefore have a conjecture.

\begin{conj}[Generic case] For almost all triangulated polyhedra $P$ with $AB=A’B’$ and $AB'=BA’$, the twinned polyhedron $T$ of $P$ has a single flex.
\end{conj}

\begin{figure}
\centering
\subfloat{
\begin{tikzpicture}[line join=bevel,z=-10, scale = 2.8]
\coordinate (A1) at (0,0,-1);
\coordinate (A2) at (-1,0,0);
\coordinate (A3) at (0,0,1);
\coordinate (A4) at (1,0,0);
\coordinate (C1) at (-0.3,1.1,0.5);
\coordinate (C2) at (0.3, 1,0.5);

\draw[fill opacity=0.7,fill=white](A4)--(C2)--(A1)--cycle;
\draw[fill opacity=0.7,fill=white](A1)--(C1)--(C2)--cycle;
\draw[fill opacity=0.7,fill=white] (A1)--(C1)--(A2)--cycle;
\draw [fill opacity=0.7,fill=white]  (A3)--(C1)--(A2)--cycle;
\draw[fill opacity=0.7,fill=white] (A3)--(C1)--(C2)--cycle;
\draw[fill opacity=0.7,fill=white] (A3)--(C2)--(A4)--cycle;
\draw[blue](A1)--(A2)--(A3)--(A4)--cycle;

\end{tikzpicture}}
\qquad
\subfloat{
\begin{tikzpicture}[line join=bevel,z=-10, scale = 2.8]
\coordinate (A1) at (0,0,-1);
\coordinate (A2) at (-1,0,0);
\coordinate (A3) at (0,0,1);
\coordinate (A4) at (1,0,0);
\coordinate (B1) at (-0.3,1,-0.5);
\coordinate (B2) at (0.3, 1.1,-0.5);
\coordinate (C1) at (-0.3,1.1,0.5);
\coordinate (C2) at (0.3, 1,0.5);

\draw (A1) -- (A2) -- (B1) -- cycle;
\draw (B1) -- (A1) -- (B2) -- cycle;
\draw (A4) -- (A1) -- (B2) -- cycle;
\draw[fill opacity=0.7,fill=white](B2) -- (B1) -- (A3) -- cycle;
\draw[fill opacity=0.7,fill=white] (B1) --(A2)--(A3)--cycle;
\draw[fill opacity=0.7,fill=white](A3)--(B2)--(A4)--cycle;
\draw(A4)--(C2)--(A1)--cycle;
\draw(A1)--(C1)--(C2)--cycle;
\draw(A1)--(C1)--(A2)--cycle;
\draw [fill opacity=0.7,fill=white]  (A3)--(C1)--(A2)--cycle;
\draw[fill opacity=0.7,fill=white] (A3)--(C1)--(C2)--cycle;
\draw[fill opacity=0.7,fill=white] (A3)--(C2)--(A4)--cycle;
\draw[blue](A1)--(A2)--(A3)--(A4)--cycle;
\end{tikzpicture}}
\qquad
\subfloat{
\begin{tikzpicture}[line join=bevel,z=-10, scale = 2.8]
\coordinate (A2) at (-1,0,0);
\coordinate (A3) at (0,0,1);
\coordinate (A4) at (1,0,0);
\coordinate (B1) at (-0.3,1,-0.5);
\coordinate (B2) at (0.3, 1.1,-0.5);
\coordinate (C1) at (-0.3,1.1,0.5);
\coordinate (C2) at (0.3, 1,0.5);

\draw[fill opacity=0.7,fill=white](B2) -- (B1) -- (A3) -- cycle;
\draw[fill opacity=0.7,fill=white] (B1) --(A2)--(A3)--cycle;
\draw[fill opacity=0.7,fill=white](A3)--(B2)--(A4)--cycle;
\draw[thick, red](A2)--(B1)--(B2)--(A4)--(C2)--(C1)--cycle;
\draw [fill opacity=0.7,fill=white]  (A3)--(C1)--(A2)--cycle;
\draw[fill opacity=0.7,fill=white] (A3)--(C1)--(C2)--cycle;
\draw[fill opacity=0.7,fill=white] (A3)--(C2)--(A4)--cycle;
\draw[thick, red](A4)--(C2)--(C1)--(A2);
\end{tikzpicture}}
\qquad
\subfloat{
\begin{tikzpicture}[line join=bevel,z=-10, scale = 3]
\coordinate (A1) at (0,0,-1);
\coordinate (A2) at (-1,0,0);
\coordinate (A3) at (0,0,1);
\coordinate (A4) at (1,0,0);
\coordinate (B1) at (-0.3,1,-0.5);
\coordinate (B2) at (0.3, 1.1,-0.5);
\coordinate (C1) at (-0.3,1.1,0.5);
\coordinate (C2) at (0.3, 1,0.5);

\draw (A1) -- (A2) -- (B1) -- cycle;
\draw (B1) -- (A1) -- (B2) -- cycle;
\draw (A4) -- (A1) -- (B2) -- cycle;
\draw[fill opacity=0.7,fill=white](B2) -- (B1) -- (A3) -- cycle;
\draw[fill opacity=0.7,fill=white] (B1) --(A2)--(A3)--cycle;
\draw[fill opacity=0.7,fill=white](A3)--(B2)--(A4)--cycle;
\draw(A4)--(C2)--(A1)--cycle;
\draw(A1)--(C1)--(C2)--cycle;
\draw(A1)--(C1)--(A2)--cycle;
\draw[thick, red](A2)--(B1)--(B2)--(A4)--(C2)--(C1)--cycle;
\draw [fill opacity=0.7,fill=white]  (A3)--(C1)--(A2)--cycle;
\draw[fill opacity=0.7,fill=white] (A3)--(C1)--(C2)--cycle;
\draw[fill opacity=0.7,fill=white] (A3)--(C2)--(A4)--cycle;
\draw[thick, red](A4)--(C2)--(C1)--(A2);
\end{tikzpicture}}
\caption{A twinned digonal anticupola and its twinned version, obtained by removing a diagonal of the square base with blue boundary and rotating a copy around its axis of symmetry. The equator is depicted in blue. Below, see how the equator may also be regarded as the red hexagonal boundary of a twinned hexagonal cone.}
\label{twinned_digonal_anticupola}
\end{figure}
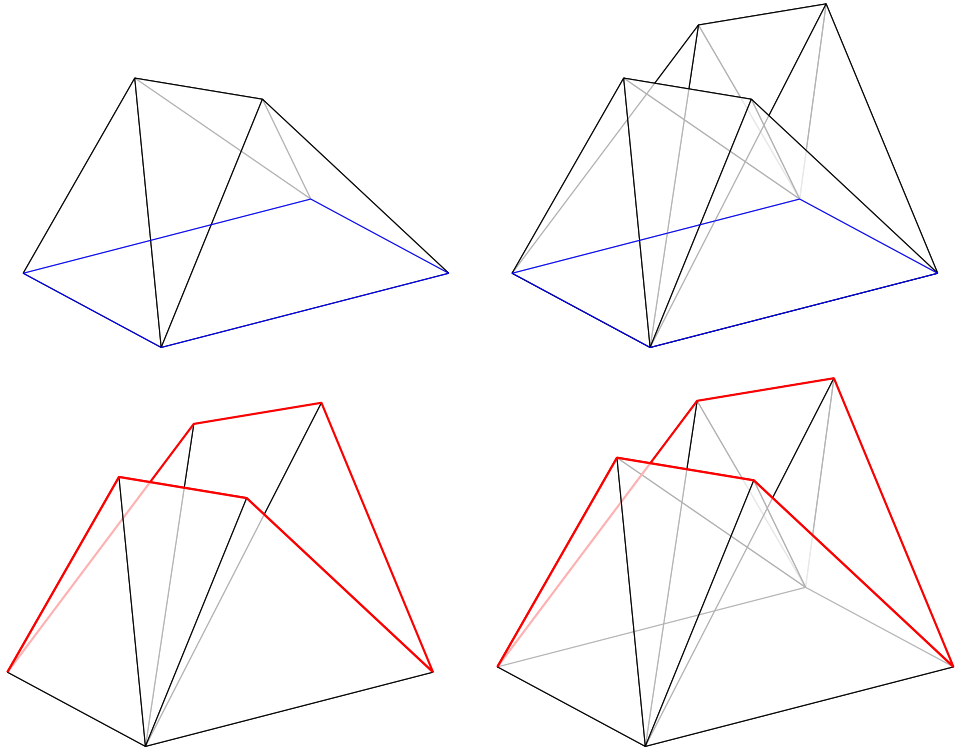

Our method works on any shape with a symmetric boundary such that the symmetry is preserved in the flexing motion. The cases here described have quadrilateral boundaries. Larger boundaries may be achieved as the following simple example shows. Take a digonal anticupola on a square, remove the square bottom, and twin it on its square boundary, as shown in figure \ref{twinned_digonal_anticupola}. This has a flexing motion, and the equator is coloured blue. Note that in the final model, the equator may not be uniquely defined, and various different boundaries of symmetry may be used. To illustrate this, we have coloured a different, possible equator in the same figure in red. Working backwards, we see that this shape may also be realised as a twinned hexagonal cone. In twinning the cone, the symmetry on the base is quite difficult to realise, and the first explanation gives a clearer view of the flexibility.

Above, we have generalised the method for constructing the Bricard type I octahedron. We get infinitely many examples, where every polyhedron $P$ with the required symmetry property of two adjacent triangular faces gives rise to a flexible polyhedron. The Bricard type I octahedron is part of this family, it is the twinned pyramid.

All examples described by this method have zero volume, as the rotated cap has a negative volume equal in size to the original cap. Also note that if $P$ was symmetric in the line $l$, the construction simply gives a double cover of $P$, but in the generic case, this will not happen. Since the polyhedron has zero volume, it must self-intersect. This can also be seen in the following way. The line $l$ goes through the polyhedron $P$ at the edge $AA'$, so it must intersect the polyhedron again at some point. This point will lie on both $P$ and $P'$, which is where it must self-intersect. We will discuss how to eliminate these self-intersections in many cases, and give an example of a novel flexible polyhedron without self-intersections.

\subsection{Type II}
In an analogous way, we can adapt the method of the Bricard type II octahedron. The construction of the type II octahedron is based on the following theorem, which has a similar proof to theorem \ref{line_symmetry} above.

\begin{theorem}
Let $A$, $B$, $A'$, $B'$ be four points in $\mathbb{R}^3$ such that $AB = AB'$ and $A'B = A'B'$. Then there exists a plane $\pi$ in $\mathbb{R}^3$ through $A$ and $A'$, such that reflecting in $\pi$ swaps $B$ and $B'$.
\end{theorem}

\begin{proof}
If the four points are coplanar, the quadrilateral $ABA'B'$ is a kite, and so a plane perpendicular to the plane of the quadrilateral will work.

\begin{figure}\centering
\begin{tikzpicture}[line join=bevel,scale = 1.3,tick/.style={
    postaction={decorate},
    decoration={
      markings,
      mark=at position 0.5 with {
        \draw[-] (-2pt,2pt) -- (2pt,-2pt);
      }
    }
  }, doubletick/.style={
    postaction={decorate},
    decoration={
      markings,
      mark=at position 0.485 with {
        \draw[-] (-2pt,2pt) -- (2pt,-2pt);
      },  mark=at position 0.515 with {
        \draw[-] (-2pt,2pt) -- (2pt,-2pt);
      }
    }
  }]
\coordinate (B) at (1.8,-1.5,1.5);
\coordinate (A) at (2,0,0);
\coordinate(B') at (1.8,-1.5,-1.5);
\coordinate (A') at (-2,0,0);
\coordinate(M) at (1.8,-1.5,0);
\coordinate(X) at (3,0.8,0);
\coordinate(Y) at (-3,0.8,0);
\coordinate(X') at (-3,-2.6,0);
\coordinate(Y') at (3,-2.6,0);
\draw[dashed](B')--(M);
\draw(A)--(A');
\draw[dashed,tick](A')--(B');
\draw[dashed, doubletick](A)--(B');
\draw(M)--(B);
\draw[tick](A')--(B);
\draw[doubletick](A)--(B);
\draw(A')--(M);
\draw(A)--(M);
\draw[fill opacity = 0.05, fill = black, dotted] (X)--(Y)--(X')--(Y')--cycle;
\draw[fill=black] (A) circle [radius=0.5pt] node[above right] {$A$};
\draw[fill=black] (B) circle [radius=0.5pt] node[below right] {$B$};
\draw[fill=black] (A') circle [radius=0.5pt] node[above left] {$A'$};
\draw[fill=black] (B') circle [radius=0.5pt] node[below right] {$B'$};
\draw[fill=black] (M) circle [radius=0.5pt] node[below right] {$M$};
\end{tikzpicture}
\caption{The quadrilateral $ABA'B'$ has reflective symmetry. $B'$ lies behind the gray plane $\pi$.}
\label{theorem2_fig}
\end{figure}

In the other case, refer to figure \ref{theorem2_fig}.  Let $M$ be the midpoint of $BB'$, and $\pi$ be the plane through $A$, $A'$, and $M$. Since $\triangle{BA'B'}$ is isosceles, $$\angle{A'MB} = \angle{B'MA'} = 90^{\circ}.$$ Similarly, since $\triangle{BAB'}$ is isosceles, we also have $$\angle{AMB} = \angle{B'MA} = 90^{\circ}.$$ 
This shows $BB' \perp \pi$, and the mirror symmetry in the plane $\pi$ follows as $M$ is the midpoint of $BB'$.
\end{proof}

To get a Bricard type II octahedron, we start with a pyramid with base $ABA'B'$ and vertex $E$ such that $AB =AB'$ and $A'B = A'B'$. Reflecting the vertex $E$ in the symmetric plane of the boundary and gluing along the common boundary creates the flexible octahedron. Again, we see that this process works for any triangulated polyhedron, giving us a type II twinning. As before, we get infinitely many self-intersecting, flexible polyhedra.

We get an analogous theorem to theorem \ref{formal_statement_I}.

\begin{theorem}
Let $P$ be a polyhedron such that the edges form an isostatic framework. Assume that there are some faces $\triangle{ABA'}$ and $\triangle{BA'B'}$ in $P$, such that $AB = AB'$ and $A'B = A'B'$. Let $T$ be a twinned polyhedron of $P$ around the equator $ABA'B'$. Then $T$  has a single flex.
\end{theorem}

\section{Examples of twinned polyhedra}
\label{section_examples}
\subsection{A new crinkle}

\begin{figure}
\centering
\subfloat{
\begin{tikzpicture}[line join=bevel,z=-10, scale = 2.5]
\coordinate (A1) at (0,0,-1);
\coordinate (A2) at (-1,0,0);
\coordinate (A3) at (0,0,1);
\coordinate (A4) at (1,0,0);
\coordinate (B1) at (-1.3,1,-0);
\coordinate (B2) at (1.3, -1,-0);

\draw (A1) -- (A2) -- (B1) -- cycle;
\draw [fill opacity=0.7,fill=white](B1) -- (A1) -- (B2) -- cycle;
\draw[fill opacity=0.7,fill=white] (A4) -- (A1) -- (B2) -- cycle;
\draw[fill opacity=0.7,fill=white](B2) -- (B1) -- (A3) -- cycle;
\draw[fill opacity=0.7,fill=white] (B1) --(A2)--(A3)--cycle;
\draw[fill opacity=0.7,fill=white](A3)--(B2)--(A4)--cycle;

\draw[blue](A1)--(A2)--(A3)--(A4)--cycle;
\end{tikzpicture}}
\qquad
\subfloat{
\begin{tikzpicture}[line join=bevel,z=-10, scale = 2.5]
\coordinate (A1) at (0,0,-1);
\coordinate (A2) at (-1,0,0);
\coordinate (A3) at (0,0,1);
\coordinate (A4) at (1,0,0);
\coordinate (B1) at (-1.3,1,-0);
\coordinate (B2) at (1.3, -1,-0);
\coordinate (C1) at (-1.3,-1,0);
\coordinate (C2) at (1.3, 1,0);

\draw (A1) -- (A2) -- (B1) -- cycle;
\draw(A1)--(C1)--(A2)--cycle;
\draw (0,0,0)--(C1);
\draw [fill opacity=0.7,fill=white](B1) -- (A1) -- (B2) -- cycle;
\draw[fill opacity=0.7,fill=white] (A4) -- (A1) -- (B2) -- cycle;
\draw[fill opacity=0.7,fill=white](B2) -- (B1) -- (A3) -- cycle;
\draw[fill opacity=0.7,fill=white] (B1) --(A2)--(A3)--cycle;
\draw[fill opacity=0.7,fill=white](A3)--(B2)--(A4)--cycle;
\draw[fill opacity=0.7,fill=white](A4)--(C2)--(A1)--cycle;
\draw [fill opacity=0.7,fill=white]  (A3)--(C1)--(A2)--cycle;
\draw[fill opacity = 0.7, fill =white,dotted] (A3)--(A1)--(C2)--cycle;
\draw (A3)--(C2);
\draw (C2)--(A1);
\draw[fill opacity=0.7,fill=white] (A3)--(C2)--(A4)--cycle;
\draw (0,0,0)--(C2);

\draw[blue](A1)--(A2)--(A3)--(A4)--cycle;
\end{tikzpicture}}
\qquad
\subfloat{
\begin{tikzpicture}[line cap=round,line join=round,>=triangle 45, scale = 0.6]
\clip(-11,-6.5) rectangle (11,6.5);
\draw [line width=1pt,dash pattern=on 4pt off 4pt] (-3,0)-- (0,6.324555320336759);
\draw [line width=1pt,dash pattern=on 4pt off 4pt] (0,6.324555320336759)-- (3,0);
\draw [line width=1pt] (3,0)-- (0,-6.324555320336759);
\draw [line width=1pt] (0,-6.324555320336759)-- (-3,0);
\draw [line width=1pt] (-3,0)-- (-4.276409071052282,3.3661223809207312);
\draw [line width=1pt] (-4.276409071052282,3.3661223809207312)-- (0,6.324555320336759);
\draw [line width=1pt] (-1.9377787402487041,4.983994031834671) -- (-1.971818318079029,4.604827896295821);
\draw [line width=1pt] (-1.9377787402487041,4.983994031834671) -- (-2.304590752973254,5.0858498049616685);
\draw [line width=1pt] (4.276409071052282,3.3661223809207312)-- (3,0);
\draw [line width=1pt,dash pattern=on 4pt off 4pt] (-4.276409071052282,-3.3661223809207312)-- (-3,0);
\draw [line width=1pt,dash pattern=on 4pt off 4pt] (3,0)-- (4.276409071052282,-3.3661223809207312);
\draw [line width=1pt] (0,6.324555320336759)-- (0,-6.324555320336759);
\draw [line width=1pt] (-4.276409071052282,-3.3661223809207312)-- (-9.439213330335313,-2.745274428255211);
\draw [line width=1pt] (-6.800433022716626,-3.258972994724889) -- (-6.753876599290562,-2.8718223242119123);
\draw [line width=1pt] (-6.961745802097032,-3.239574484964031) -- (-6.915189378670967,-2.852423814451054);
\draw [line width=1pt] (-9.439213330335313,-2.745274428255211)-- (-3,0);
\draw [line width=1pt] (-4.276409071052282,3.3661223809207312)-- (-9.439213330335313,2.745274428255211);
\draw [line width=1pt] (-6.834532988980764,2.862123069331484) -- (-6.881089412406829,3.249273739844459);
\draw [line width=1pt] (-9.439213330335313,2.745274428255211)-- (-3,0);
\draw [line width=1pt] (4.276409071052282,-3.3661223809207312)-- (9.439213330335313,-2.745274428255211);
\draw [line width=1pt] (6.753876599290564,-2.8718223242119123) -- (6.800433022716628,-3.258972994724889);
\draw [line width=1pt] (6.915189378670973,-2.852423814451054) -- (6.961745802097037,-3.239574484964031);
\draw [line width=1pt] (9.439213330335313,-2.745274428255211)-- (3,0);
\draw [line width=1pt] (3,0)-- (9.439213330335313,2.745274428255211);
\draw [line width=1pt] (9.439213330335313,2.745274428255211)-- (4.276409071052282,3.3661223809207312);
\draw [line width=1pt] (6.834532988980769,2.862123069331484) -- (6.8810894124068325,3.249273739844459);
\draw [line width=1pt] (4.276409071052282,-3.3661223809207312)-- (0,-6.324555320336759);
\draw [line width=1pt] (1.7373529449712684,-5.122649213040598) -- (1.7713925228015917,-4.743483077501747);
\draw [line width=1pt] (1.7373529449712684,-5.122649213040598) -- (2.1041649576958164,-5.224504986167594);
\draw [line width=1pt] (2.1382045355261434,-4.845338850628745) -- (2.1722441133564665,-4.466172715089894);
\draw [line width=1pt] (2.1382045355261434,-4.845338850628745) -- (2.5050165482506914,-4.947194623755742);
\draw [line width=1pt] (4.276409071052282,3.3661223809207312)-- (0,6.324555320336759);
\draw [line width=1pt] (1.9377787402487077,4.983994031834671) -- (2.304590752973256,5.0858498049616685);
\draw [line width=1pt] (1.9377787402487077,4.983994031834671) -- (1.9718183180790307,4.604827896295821);
\draw [line width=1pt] (-4.276409071052282,-3.3661223809207312)-- (0,-6.324555320336759);
\draw [line width=1pt] (-1.7373529449712684,-5.122649213040598) -- (-2.104164957695818,-5.224504986167594);
\draw [line width=1pt] (-1.7373529449712684,-5.122649213040598) -- (-1.7713925228015934,-4.743483077501747);
\draw [line width=1pt] (-2.1382045355261416,-4.845338850628745) -- (-2.5050165482506914,-4.947194623755742);
\draw [line width=1pt] (-2.1382045355261416,-4.845338850628745) -- (-2.1722441133564665,-4.466172715089894);
\end{tikzpicture}}
\caption{A flexible dodecahedron constructed by twinning. The original polyhedron is self-intersecting, as the two peaks of the digonal anticupola lie on different sides of the quadrilateral base. The equator is drawn in blue, and the polyhedron self-intersects along the dotted line. On the bottom, a net for assembling a crinkle obtained from this polyhedron. Labels show gluing instructions.}
\label{crinkle_twinning}
\end{figure}
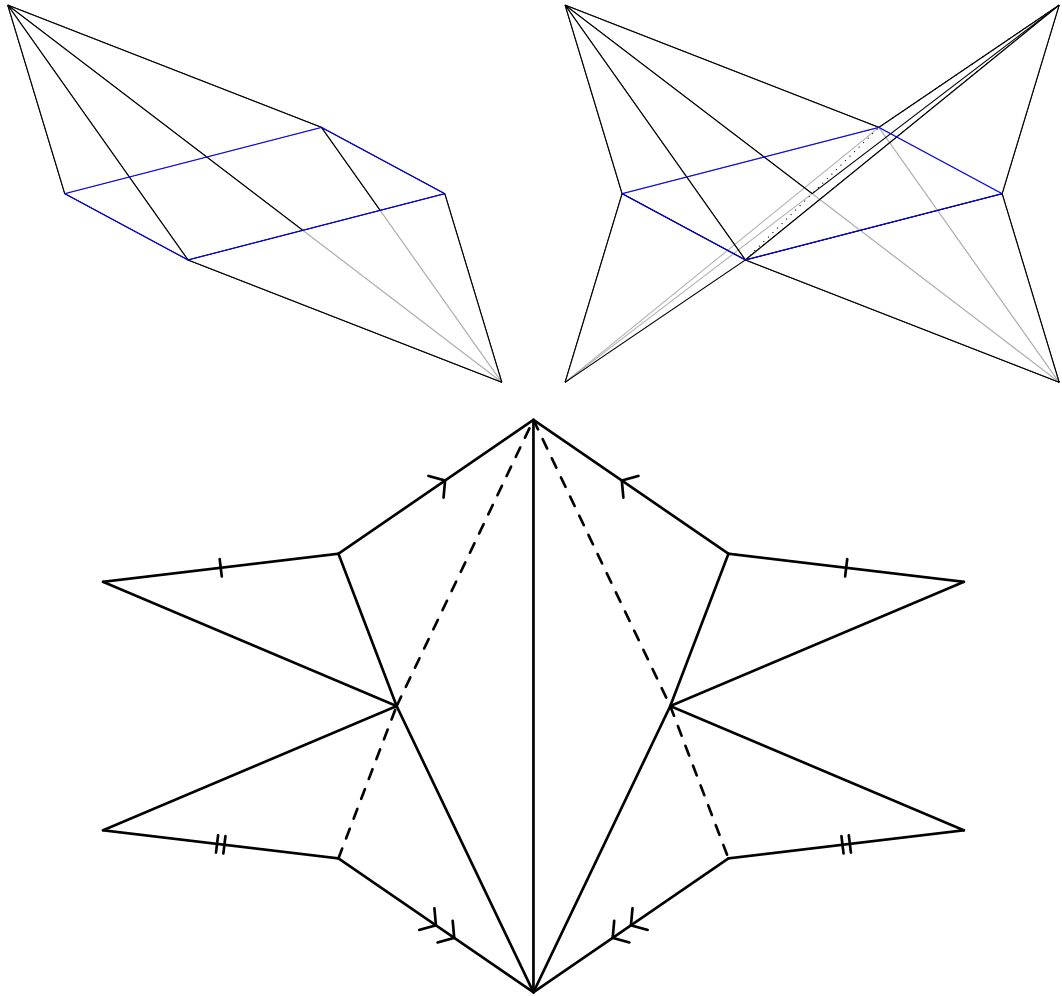

We next focus on a particular case. Let the starting shape be a digonal anticupola, as depicted in the top left corner of figure \ref{twinned_digonal_anticupola}. We realise the square face by adding a diagonal in that square as an edge, and when twinning, we remove that edge. The twinning action results in a flexible, triangulated dodecahedron. This flexible dodecahedron has been described before in \cite{flexible_suspensions}, but not as a part of this larger family of twinned polyhedra.

If we start out with the two vertices of the digonal anticupola on different sides of the base, we get a star-shaped flexible dodecahedron, as seen in figure \ref{crinkle_twinning}. This dodecahedron has a self-intersection from the rotated edge going through its non-rotated version. Removing this edge, we get a new crinkle, with a net given on the bottom in figure \ref{crinkle_twinning}. This crinkle looks similar to the Bricard crinkle, but has an indentation on one of the sides, meaning it can be constructed to take less space. The parameters are chosen so that the crinkle can lie flat.

\subsection{Methods of removing self-intersections}

One type of self-intersection is a vertex passing through a face. We can fix this by removing the face, adding a vertex above the face, and adding a cone on the face with the new vertex as apex. This is called \textit{erecting a tent}.

Another type of self-intersection is an edge lying below another edge. This can often be remedied by swapping out edges for crinkles. The Bricard crinkle is the basis of most known methods of constructing flexible polyhedra, as discussed in \cite{Li}. It is a surface with the topology of a disk that has a range of motion such that two of the vertices keep a fixed distance throughout the motion, as if they were connected by a phantom edge. This gadget may be used to remove double covers of an edge, or a hinge participating in a flexing motion. This is done by swapping the edge for two opposing crinkles, arranged such that they do not intersect. We see this implemented when Steffen's polyhedron is constructed by adding crinkles to the system of a tetrahedron and a triangle glued together along an edge. The edge corresponds to a hinge, and the crinkles take the place of it. This can be considered as an example of the base+crinkle method, as discussed in \cite{Zeyuan_crinkle_base_methods}. Here, you start out with a flexible base that is not a polyhedron, and turn it into a polyhedron by adding crinkles. The crinkles all have zero volume, so the base must have positive volume in order for the result to stand a chance of being embedded.

To summarise, the two potential issues are that of vertices sticking through faces, and edges meeting edges. The first problem can be resolved by erecting a tent on the face, and the second by swapping the edges for crinkles. This is the idea of Connelly's methods as described in \cite{p1}. We may apply this method using crinkles arising from our twinned polyhedra, giving fundamentally different constructions of flexible polyhedra. We can also use our twinned polyhedra to obtain new crinkles by removing edges in the right way.

Some trial and error is necessary to find a good choice of parameters.

\subsection{A novel flexible polyhedron}
\label{foxtrot_section}

\begin{figure}
\centering
\subfloat{
\begin{tikzpicture}[line join=bevel,scale = 3]
\coordinate (A2) at (0,0,-1);
\coordinate (A3) at (-1,0,0);
\coordinate (A4) at (0,0,1);
\coordinate (A1) at (1,0,0);
\coordinate (B1) at (0.2,0.7,-0.4);
\coordinate (B2) at (0.2, 0.7,0.4);
\coordinate (C1) at (-0.2,0.7,-0.4);
\coordinate (C2) at (-0.2, 0.7,0.4);

\draw[fill opacity=0.7,fill=white] (A3)--(C1)--(C2)--cycle;
\draw [fill opacity=0.7,fill=white]  (A3)--(C1)--(A2)--cycle;
\draw [fill opacity=0.7,fill=white]  (A1)--(B1)--(A2)--cycle;
\draw [fill opacity=0.7,fill=white]  (A2)--(B1)--(A3)--cycle;
\draw[blue,thick](C1)--(A2)--(A1);

\draw[fill opacity=0.7,fill=white](B2) -- (B1) -- (A3) -- cycle;
\draw[fill opacity=0.7,fill=white](A3)--(B2)--(A4)--cycle;

\draw[fill opacity=0.7,fill=white] (A3)--(C2)--(A4)--cycle;
\draw[fill opacity=0.7,fill=white] (A4) -- (A1) -- (B2) -- cycle;
\draw[fill opacity=0.7,fill=white] (B1) -- (A1) -- (B2) -- cycle;

\draw[blue,thick](C1)--(C2)--(A4)--(A1);
\draw[dashed](C1)--(A1);
\draw[dashed](C2)--(A1);

\end{tikzpicture}}
\qquad
\subfloat{
\begin{tikzpicture}[line join=bevel, scale = 3]
\coordinate (A2) at (0,0,-1);
\coordinate (A4) at (0,0,1);
\coordinate (A1) at (1,0,0);
\coordinate (C1) at (-0.2,0.7,-0.4);
\coordinate (C2) at (-0.2, 0.7,0.4);

\coordinate (X1) at (0.5,0.7,-1);
\coordinate (X2) at (1,1.3,1);
\coordinate (X3) at (0.5,0.7,0);

\draw[blue](A2)--(C1)--(C2)--(A4)--(A1)--cycle;

\draw(A2)--(X1)--(A1)--cycle;
\draw(A2)--(X1)--(C1)--cycle;
\draw[fill opacity=0.6,fill=white] (A1)--(X1)--(C1)--cycle;
\draw[blue, thick](C1)--(C2)--(A4)--(A1)--(A2)--cycle;
\draw(A1)--(C1)--(X2)--cycle;
\draw[fill opacity=0.6,fill=white] (A1)--(C2)--(X2)--cycle;
\draw[fill opacity=0.6,fill=white](C1)--(C2)--(X2)--cycle;

\draw(A1)--(C2)--(X3)--cycle;
\draw[fill opacity=0.6,fill=white](A4)--(C2)--(X3)--cycle;
\draw[fill opacity=0.6,fill=white] (A4)--(A1)--(X3)--cycle;
\draw[blue, thick](C1)--(C2)--(A4)--(A1);
\draw[dashed](C1)--(A1);
\draw[dashed](C2)--(A1);

\end{tikzpicture}}

\caption{Left: A  twinned digonal anticupola with two edges removed gives a pentagonal crinkle without self-intersections. The boundary is coloured blue, and the removed edges are dashed. Right: a tetrahedral chain as base to glue the pentagonal crinkle onto, gluing blue on blue, and so that the dashed lines meet.}
\label{Gluing_crinkle}
\end{figure}
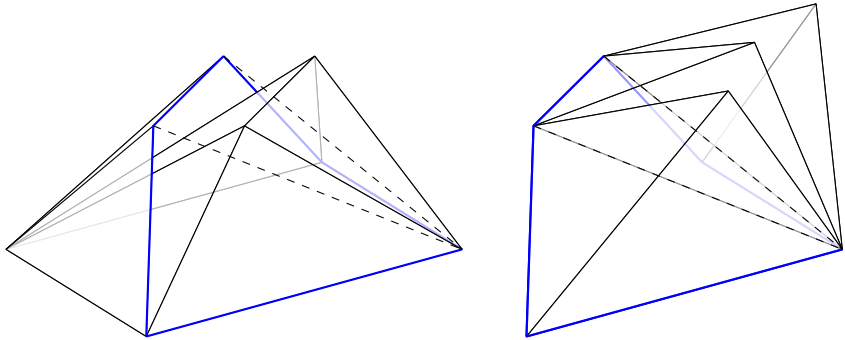

Starting with a twinned digonal anticupola, we may remove two edges to create a crinkle, as seen in figure \ref{Gluing_crinkle}. Whereas the original Bricard crinkle has a quadrilateral boundary and has a movement like two triangles glued together on an edge, this crinkle of pentagonal boundary may be regarded as three triangles glued together two at an edge, all flexing in a joined motion. Also note that in creating it we have removed two edges, so the system itself has two degrees of freedom, but we will only use one of them.

To make a flexible polyhedron, we glue the pentagonal crinkle depicted in figure \ref{Gluing_crinkle} to three linked tetrahedral caps as the figure shows. The tetrahedra are not closed, their bottoms are missing, and these holes align with the holes in the crinkle. We get rid of further self-intersections by adding a Bricard crinkle between each of the two adjacent tetrahedra. Finally, we dispose of any additional vertices poking through faces by adding tents on those faces. By trial and error, we found a selection of edge lengths without self-intersection that showed a reasonable range of motion. A picture of the assembled model can be see in figure \ref{foxtrot}, and we also  include a net for assembling the polyhedron in figure \ref{net}. The model bears a resemblance to a fox, and the working name for the polyhedron was the foxtrot, after a popular American dance. A three dimensional version, showing the flexing motion, is available online at \url{https://www.geogebra.org/m/fq3vdnmy}. 

This construction bears a likeness to that of Steffen's polyhedron, and is an example of the base+crinkle method outlined in \cite{Zeyuan_crinkle_base_methods}. Here, the base is three linked tetrahedra, and the crinkles are two Bricard crinkles and one pentagonal crinkle. Note that the volume of the polyhedron comes from the tetrahedra, and remains constant during the flexing motion as predicted by the Bellows theorem.

\begin{figure}
\centering
\definecolor{ccqqqq}{rgb}{0.8,0,0}
\definecolor{qqqqff}{rgb}{0,0,1}
\begin{tikzpicture}[line cap=round,line join=round,>=triangle 45, scale = 1.2]
\clip(-6.5,-6) rectangle (7,6);
\draw [line width=1pt,dash pattern=on 4pt off 4pt] (-1.5,0)-- (0,2);
\draw [line width=1pt,dash pattern=on 4pt off 4pt] (0,2)-- (1.5,0);
\draw [line width=1pt,dash pattern=on 4pt off 4pt] (1.5,0)-- (-1.5,0);
\draw [line width=1pt,color=qqqqff] (-1.5,0)-- (-2.2443880443810174,0.8987632832857632);
\draw [line width=1pt,color=qqqqff] (-1.9598343540106158,0.3767947698539381) -- (-1.7845536903704005,0.5219685134318242);
\draw [line width=1pt] (-2.2443880443810174,0.8987632832857632)-- (0,2);
\draw [line width=1pt] (0,2)-- (2.2443880443810174,0.8987632832857632);
\draw [line width=1pt,color=ccqqqq] (2.2443880443810174,0.8987632832857632)-- (1.5,0);
\draw [line width=1pt,color=ccqqqq] (1.9598343540106158,0.3767947698539381) -- (1.7845536903704005,0.5219685134318242);
\draw [line width=1pt,dash pattern=on 4pt off 4pt] (-2.2443880443810174,0.8987632832857632)-- (-1.1334238941470451,4.2283065938457725);
\draw [line width=1pt,dash pattern=on 4pt off 4pt] (-1.1334238941470451,4.2283065938457725)-- (0,2);
\draw [line width=1pt,dash pattern=on 4pt off 4pt] (2.2443880443810174,0.8987632832857632)-- (1.1334238941470451,4.2283065938457725);
\draw [line width=1pt,dash pattern=on 4pt off 4pt] (1.1334238941470451,4.2283065938457725)-- (0,2);
\draw [line width=1pt,color=ccqqqq] (0.0015086767192895522,4.49999954477887)-- (1.1334238941470451,4.2283065938457725);
\draw [line width=1pt,color=ccqqqq] (0.7057834060871914,4.330952891333921) -- (0.527626071859088,4.198172524527493);
\draw [line width=1pt,color=ccqqqq] (0.7057834060871914,4.330952891333921) -- (0.607306499007248,4.530133614097149);
\draw [line width=1pt] (0.0015086767192895522,4.49999954477887)-- (0,2);
\draw [line width=1pt,dash pattern=on 4pt off 4pt] (-1.1334238941470451,4.2283065938457725)-- (-3.3778119385280627,3.1270698771315355);
\draw [line width=1pt] (-3.3778119385280627,3.1270698771315355)-- (-2.2443880443810174,0.8987632832857632);
\draw [line width=1pt] (2.2443880443810174,0.8987632832857632)-- (3.3778119385280627,3.1270698771315355);
\draw [line width=1pt,dash pattern=on 4pt off 4pt] (3.3778119385280627,3.1270698771315355)-- (1.1334238941470451,4.2283065938457725);
\draw [line width=1pt,color=qqqqff] (-2.2443880443810174,0.8987632832857632)-- (-3.379320615247351,0.6270703323526654);
\draw [line width=1pt,color=qqqqff] (-2.7853609685300977,0.6522471031773669) -- (-2.83834769109827,0.8735865124610628);
\draw [line width=1pt,dash pattern=on 4pt off 4pt] (-3.379320615247351,0.6270703323526654)-- (-3.3778119385280627,3.1270698771315355);
\draw [line width=1pt] (-3.3778119385280627,3.1270698771315355)-- (-1.8778119385280638,5.127069877131537);
\draw [line width=1pt,color=qqqqff] (-1.8778119385280638,5.127069877131537)-- (-1.1334238941470451,4.2283065938457725);
\draw [line width=1pt,color=qqqqff] (-1.414884326601375,4.56813782071352) -- (-1.6370784140677155,4.56880792780524);
\draw [line width=1pt,color=qqqqff] (-1.414884326601375,4.56813782071352) -- (-1.3741574186073924,4.786568543172069);
\draw [line width=1pt] (1.8778119385280638,5.127069877131537)-- (3.3778119385280627,3.1270698771315355);
\draw [line width=1pt,dash pattern=on 4pt off 4pt] (3.3778119385280627,3.1270698771315355)-- (3.379320615247351,0.6270703323526654);
\draw [line width=1pt,color=ccqqqq] (3.379320615247351,0.6270703323526654)-- (2.2443880443810174,0.8987632832857632);
\draw [line width=1pt,color=ccqqqq] (2.7853609685300977,0.6522471031773669) -- (2.8383476910982677,0.8735865124610628);
\draw [line width=1pt] (-3.379320615247351,0.6270703323526654)-- (-4.7477028206640774,-0.2021040400798518);
\draw [line width=1pt] (-3.9234356381746216,0.1643042283720446) -- (-4.041382239459487,0.3589508983048242);
\draw [line width=1pt] (-4.004538417313281,0.11515981117001697) -- (-4.122485018598146,0.30980648110279657);
\draw [line width=1pt] (-4.085641196451939,0.06601539396798935) -- (-4.203587797736805,0.26066206390076896);
\draw [line width=1pt] (-4.7477028206640774,-0.2021040400798518)-- (-3.3778119385280627,3.1270698771315355);
\draw [line width=1pt] (-4.008629248231338,1.5940276754672704) -- (-3.904903671266356,1.3975291608881621);
\draw [line width=1pt] (-4.008629248231338,1.5940276754672704) -- (-4.220611087925785,1.5274366761635219);
\draw [line width=1pt] (-4.116885510960803,1.3309381615844138) -- (-4.013159933995821,1.1344396470053055);
\draw [line width=1pt] (-4.116885510960803,1.3309381615844138) -- (-4.32886735065525,1.2643471622806652);
\draw [line width=1pt] (-3.900372985501871,1.857117189350127) -- (-3.7966474085368893,1.660618674771019);
\draw [line width=1pt] (-3.900372985501871,1.857117189350127) -- (-4.112354825196318,1.7905261900463785);
\draw [line width=1pt,color=qqqqff] (3.379320615247351,0.6270703323526654)-- (4.745041552661942,-0.20783263409824482);
\draw [line width=1pt,color=qqqqff] (4.040626268935121,0.3561722734927234) -- (3.921917108598646,0.16198972504209272);
\draw [line width=1pt,color=qqqqff] (4.1215356641228835,0.30671012335252545) -- (4.002826503786409,0.11252757490189477);
\draw [line width=1pt,color=qqqqff] (4.202445059310646,0.2572479732123275) -- (4.083735898974172,0.06306542476169684);
\draw [line width=1pt] (4.745041552661942,-0.20783263409824482)-- (3.3778119385280627,3.1270698771315355);
\draw [line width=1pt] (3.974121506115451,1.3725802075341016) -- (4.1847044566954255,1.4589141394242);
\draw [line width=1pt] (3.938149034494576,1.4603231036090907) -- (4.1487319850745505,1.546657035499189);
\draw [line width=1pt] (-1.5,0)-- (0,-2.598076211353316);
\draw [line width=1pt] (0,-2.598076211353316)-- (1.5,0);
\draw [line width=1pt] (0,-2.598076211353316)-- (-3.0985357137665965,-0.06843662625172908);
\draw [line width=1pt] (-3.0985357137665965,-0.06843662625172908)-- (-1.5,0);
\draw [line width=1pt] (-2.3988790394169284,0.07541803375798443) -- (-2.389144214344955,-0.1519670142363599);
\draw [line width=1pt] (-2.3041352694192847,0.079474210871307) -- (-2.2944004443473114,-0.14791083712303732);
\draw [line width=1pt] (-2.209391499421641,0.08353038798462956) -- (-2.1996566743496677,-0.14385466000971475);
\draw [line width=1pt,color=qqqqff] (1.5,0)-- (3.0985357137665965,-0.06843662625172908);
\draw [line width=1pt,color=qqqqff] (2.20939149942164,0.08353038798462956) -- (2.199656674349667,-0.14385466000971475);
\draw [line width=1pt,color=qqqqff] (2.3041352694192847,0.079474210871307) -- (2.2944004443473114,-0.14791083712303732);
\draw [line width=1pt,color=qqqqff] (2.3988790394169293,0.07541803375798443) -- (2.389144214344956,-0.1519670142363599);
\draw [line width=1pt] (0,-2.598076211353316)-- (3.0985357137665965,-0.06843662625172908);
\draw [line width=1pt,dash pattern=on 4pt off 4pt] (-3.0985357137665965,-0.06843662625172908)-- (-3.3778099341942704,-1.6438749493239184);
\draw [line width=1pt,dash pattern=on 4pt off 4pt] (-3.3778099341942704,-1.6438749493239184)-- (0,-2.598076211353316);
\draw [line width=1pt,dash pattern=on 4pt off 4pt] (-1.6579690911675815,-2.0114646069244686) -- (-1.7198408430266885,-2.230486553752766);
\draw [line width=1pt,dash pattern=on 4pt off 4pt] (3.0985357137665965,-0.06843662625172908)-- (3.3734095313113426,-1.647949113119546);
\draw [line width=1pt,dash pattern=on 4pt off 4pt] (3.3734095313113426,-1.647949113119546)-- (0,-2.598076211353316);
\draw [line width=1pt,dash pattern=on 4pt off 4pt] (-3.3778099341942704,-1.6438749493239184)-- (-6.239019285519133,-2.545806791477837);
\draw [line width=1pt,color=qqqqff] (-6.239019285519133,-2.545806791477837)-- (-3.0985357137665965,-0.06843662625172908);
\draw [line width=1pt,color=qqqqff] (-4.445417145057547,-1.1309239145119867) -- (-4.451378645738528,-1.3530390244395754);
\draw [line width=1pt,color=qqqqff] (-4.445417145057547,-1.1309239145119867) -- (-4.662815998961884,-1.085006598937195);
\draw [line width=1pt,color=qqqqff] (-4.668777499642864,-1.3071217088647837) -- (-4.674739000323845,-1.5292368187923722);
\draw [line width=1pt,color=qqqqff] (-4.668777499642864,-1.3071217088647837) -- (-4.886176353547201,-1.2612043932899921);
\draw [line width=1pt,color=qqqqff] (-6.239019285519133,-2.545806791477837)-- (-5.444722241093987,-3.8182423321703625);
\draw [line width=1pt,color=qqqqff] (-5.77044602355768,-3.081543733299535) -- (-5.963511257111676,-3.2020615430345005);
\draw [line width=1pt,color=qqqqff] (-5.720230269501445,-3.1619875806136983) -- (-5.913295503055441,-3.282505390348664);
\draw [line width=1pt] (-5.444722241093987,-3.8182423321703625)-- (-3.3778099341942704,-1.6438749493239184);
\draw [line width=1pt,color=ccqqqq] (3.0985357137665965,-0.06843662625172908)-- (6.2335282200037705,-2.5533336096691106);
\draw [line width=1pt,color=ccqqqq] (4.888981639444667,-1.4876022768355188) -- (4.671476507839865,-1.53301350093366);
\draw [line width=1pt,color=ccqqqq] (4.888981639444667,-1.4876022768355188) -- (4.883537098489985,-1.2654738938622778);
\draw [line width=1pt,color=ccqqqq] (4.666031966885182,-1.31088511796042) -- (4.448526835280379,-1.3562963420585612);
\draw [line width=1pt,color=ccqqqq] (4.666031966885182,-1.31088511796042) -- (4.660587425930499,-1.088756734987179);
\draw [line width=1pt,dash pattern=on 4pt off 4pt] (6.2335282200037705,-2.5533336096691106)-- (3.3734095313113426,-1.647949113119546);
\draw [line width=1pt,color=ccqqqq] (5.437695999285861,-3.8248095537857116)-- (6.2335282200037705,-2.5533336096691106);
\draw [line width=1pt,color=ccqqqq] (5.713995887894266,-3.1688878021928537) -- (5.90691552292392,-3.2896385425243233);
\draw [line width=1pt,color=ccqqqq] (5.7643086963657115,-3.0885046209305003) -- (5.957228331395365,-3.20925536126197);
\draw [line width=1pt,dash pattern=on 4pt off 4pt] (0,-2.598076211353316)-- (2.8783733720279816,-4.606823764233036);
\draw [line width=1pt,dash pattern=on 4pt off 4pt] (1.504311694272226,-3.509131098938839) -- (1.3740616777557575,-3.695768876647513);
\draw [line width=1pt] (2.8783733720279816,-4.606823764233036)-- (3.3734095313113426,-1.647949113119546);
\draw [line width=1pt] (3.1493637290103518,-2.9870905731625164) -- (3.2942464902861905,-3.1555531714851184);
\draw [line width=1pt] (3.1493637290103518,-2.9870905731625164) -- (2.9575364130531328,-3.0992197058674624);
\draw [line width=1pt] (-5.444722241093987,-3.8182423321703625)-- (-4.505589579804996,-5.220271609503847);
\draw [line width=1pt] (-4.505589579804996,-5.220271609503847)-- (-1.708240758569396,-4.1363714423612254);
\draw [line width=1pt] (-2.841641022564259,-4.575534669434525) -- (-2.9126059819769217,-4.786092585657291);
\draw [line width=1pt] (-2.841641022564259,-4.575534669434525) -- (-3.035950209774532,-4.4677636097097695);
\draw [line width=1pt] (-3.106915169187195,-4.678321525932534) -- (-3.1778801285998575,-4.8888794421553);
\draw [line width=1pt] (-3.106915169187195,-4.678321525932534) -- (-3.301224356397468,-4.570550466207778);
\draw [line width=1pt,color=qqqqff] (-5.444722241093987,-3.8182423321703625)-- (-6.00555057674283,-5.209454601356255);
\draw [line width=1pt,color=qqqqff] (-5.601864973689079,-4.512418919394543) -- (-5.812952068542918,-4.427325057942975);
\draw [line width=1pt,color=qqqqff] (-5.6373207492939,-4.600371875583641) -- (-5.848407844147738,-4.515278014132073);
\draw [line width=1pt,color=ccqqqq] (-6.00555057674283,-5.209454601356255)-- (-4.505589579804996,-5.220271609503847);
\draw [line width=1pt,color=ccqqqq] (-5.302163497515396,-5.100727468433687) -- (-5.303804750166382,-5.32831488715517);
\draw [line width=1pt,color=ccqqqq] (-5.2073354063814445,-5.10141132370493) -- (-5.20897665903243,-5.328998742426413);
\draw [line width=1pt] (-3.3778099341942704,-1.6438749493239184)-- (-1.708240758569396,-4.1363714423612254);
\draw [line width=1pt] (-2.463862259009657,-3.00830561112303) -- (-2.684844244718382,-2.985118900689181);
\draw [line width=1pt] (-2.463862259009657,-3.00830561112303) -- (-2.4012064480452815,-2.7951274909959625);
\draw [line width=1pt] (5.437695999285861,-3.8248095537857116)-- (3.3734095313113426,-1.647949113119546);
\draw [line width=1pt] (4.209795329893984,-2.529946445522737) -- (4.43153378035423,-2.5157084282449116);
\draw [line width=1pt] (4.209795329893984,-2.529946445522737) -- (4.183814314838358,-2.7506173507304545);
\draw [line width=1pt] (4.405552765298602,-2.7363793334526294) -- (4.627291215758846,-2.722141316174804);
\draw [line width=1pt] (4.405552765298602,-2.7363793334526294) -- (4.3795717502429765,-2.9570502386603468);
\draw [line width=1pt,color=qqqqff] (-1.8778119385280638,5.127069877131537)-- (-5.826445920962462,5.766045363957334);
\draw [line width=1pt,color=qqqqff] (-4.132967299471872,5.492003421500843) -- (-3.9652806340347224,5.637783542858604);
\draw [line width=1pt,color=qqqqff] (-4.132967299471872,5.492003421500843) -- (-4.019815595182412,5.300777499186674);
\draw [line width=1pt,color=qqqqff] (-3.8521289297452634,5.446557620544436) -- (-3.6844422643081134,5.592337741902195);
\draw [line width=1pt,color=qqqqff] (-3.8521289297452634,5.446557620544436) -- (-3.738977225455803,5.255331698230266);
\draw [line width=1pt,color=qqqqff] (0.0015086767192895522,4.49999954477887)-- (-1.1334238941470451,4.2283065938457725);
\draw [line width=1pt,color=qqqqff] (-0.7042947395161884,4.331036367707213) -- (-0.6056976506400066,4.530157626275093);
\draw [line width=1pt,color=qqqqff] (-0.7042947395161884,4.331036367707213) -- (-0.5262175667877486,4.198148512349549);
\draw [line width=1pt,color=ccqqqq] (1.8778119385280638,5.127069877131537)-- (1.1334238941470451,4.2283065938457725);
\draw [line width=1pt,color=ccqqqq] (1.414884326601375,4.5681378207135195) -- (1.3741574186073935,4.786568543172068);
\draw [line width=1pt,color=ccqqqq] (1.414884326601375,4.5681378207135195) -- (1.6370784140677164,4.56880792780524);
\draw [line width=1pt,color=ccqqqq] (1.8778119385280638,5.127069877131537)-- (5.830987072392624,5.759010443570854);
\draw [line width=1pt,color=ccqqqq] (4.135324428834794,5.48794782297494) -- (4.021806564721817,5.296939037638397);
\draw [line width=1pt,color=ccqqqq] (4.135324428834794,5.48794782297494) -- (3.967917369573321,5.634048945687738);
\draw [line width=1pt,color=ccqqqq] (3.8543995054603437,5.443040160351194) -- (3.7408816413473667,5.252031375014651);
\draw [line width=1pt,color=ccqqqq] (3.8543995054603437,5.443040160351194) -- (3.6869924461988712,5.5891412830639915);
\draw [line width=1pt] (5.830987072392624,5.759010443570854)-- (3.3778119385280627,3.1270698771315355);
\draw [line width=1pt] (4.7199722874040795,4.400135353825997) -- (4.553484805648035,4.55531475094141);
\draw [line width=1pt] (4.655314205272655,4.330765569760978) -- (4.48882672351661,4.485944966876392);
\draw [line width=1pt] (-5.826445920962462,5.766045363957334)-- (-3.3778119385280627,3.1270698771315355);
\draw [line width=1pt] (-4.505376711306744,4.3422844892447445) -- (-4.727256687304892,4.330454958418213);
\draw [line width=1pt] (-4.505376711306744,4.3422844892447445) -- (-4.477001172185634,4.562660282670657);
\draw [line width=1pt] (-4.698881148183782,4.550830751844125) -- (-4.9207611241819285,4.539001221017593);
\draw [line width=1pt] (-4.698881148183782,4.550830751844125) -- (-4.670505609062671,4.7712065452700365);
\draw [line width=1pt] (-4.3118722744297076,4.133738226645364) -- (-4.533752250427854,4.121908695818832);
\draw [line width=1pt] (-4.3118722744297076,4.133738226645364) -- (-4.2834967353085975,4.354114020071276);
\end{tikzpicture}
\caption{A net for a flexible polyhedron constructed by adding tents and crinkles on a flexible dodecahedron. Dashed lines correspond to valley folds and solid lines to mountain folds. Gluing instructions are indicated by a symbol and colour.}
\label{net}
\end{figure}
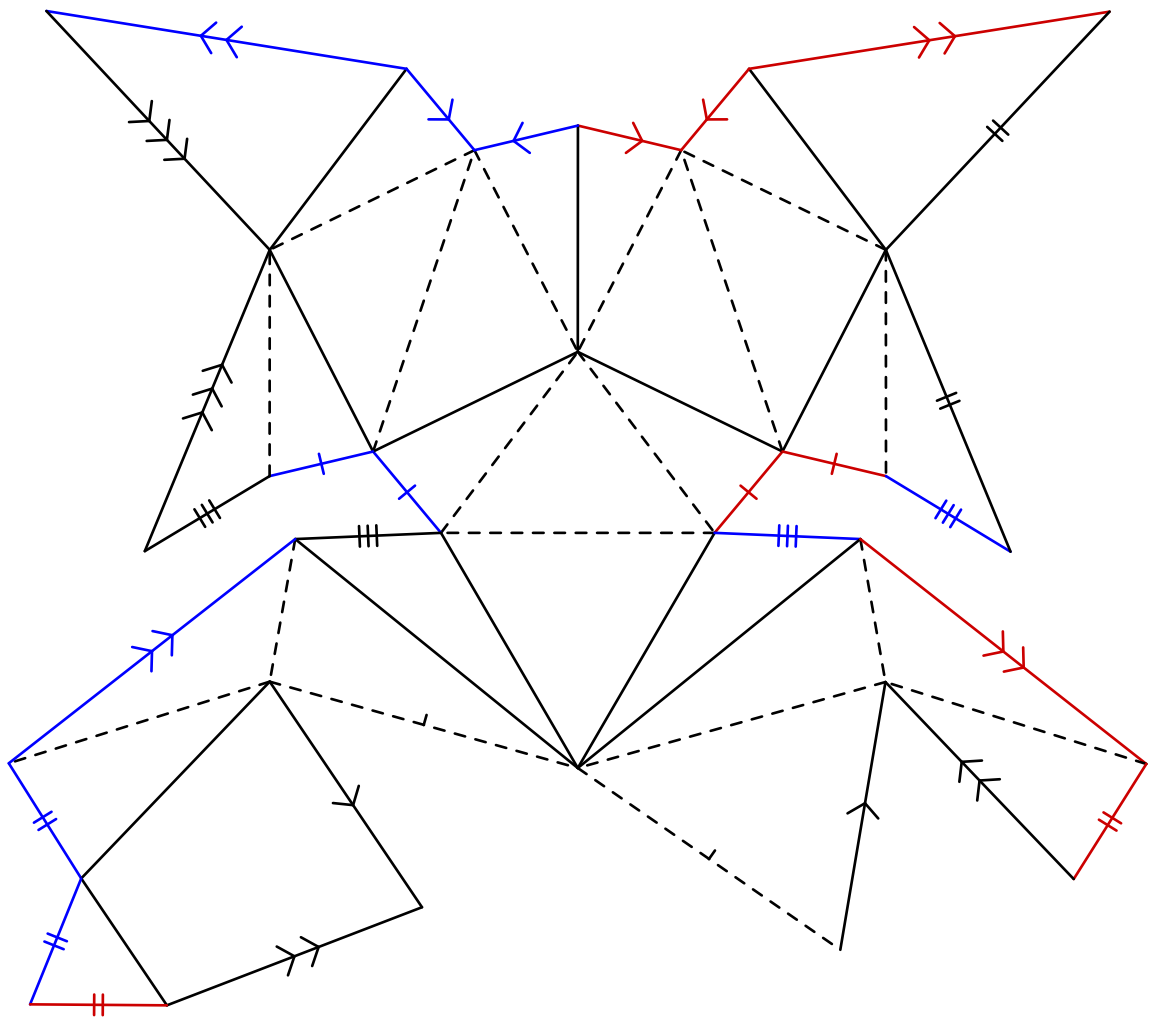

\section*{Conclusion}
We have generalised Raoul Bricard's 1897 construction of flexible octahedra, creating infinitely many flexible polyhedra with self-intersection. Using these, Connelly's methods from \cite{p1} may be used to remove the self-intersections in many cases. We described two explicit examples of notable flexible surfaces without self-intersections. Many further examples of twinned polyhedra could be considered. Each such example leads to a crinkle by the removal of an edge.

We have only generalised the simpler Bricard octahedra, types I and II, but further research could consider possible generalisations of the type III Bricard octahedron in a similar way.

\section*{Acknowledgements}

This project began when the two authors worked together on a research project at the Engineering Department at the University of Cambridge in the summer of 2022. The project was funded by the Cambridge Mathematics Department and a CUED–MathWorks Grant. Our collaborator Zeyuan He provided helpful discussion and advice throughout.

\bibliographystyle{alphaurl}
\bibliography{bibliography_flexible_polyhedra}

\end{document}